\documentclass[12pt]{amsart}
\usepackage{mathrsfs, amssymb,amsthm, amsfonts, amsmath}
\usepackage[all]{xy}
\usepackage{upgreek}

\usepackage{setspace}
\newtheorem{theorem}[equation]{Theorem}
\newtheorem*{theorem*}{Theorem}
\newtheorem{lemma}[equation]{Lemma}
\newtheorem{corollary}[equation]{Corollary}

\newtheorem{question}[equation]{Question}
\newtheorem{proposition}[equation]{Proposition}

\theoremstyle{definition}
\newtheorem{example}[equation]{Example}
\newtheorem{definition}[equation]{Definition}
\newtheorem*{definition*}{Definition}

\theoremstyle{remark}
\newtheorem{remark}[equation]{Remark}

\makeatletter\@addtoreset{equation}{section}

\makeatother

\usepackage{hyperref, tikz}

\def \KK {\mathbb{K}}

\newcommand{\CC}{\mathbb{C}}

\newcommand{\QQ}{\mathbb{Q}}

\newcommand{\ZZ}{\mathbb{Z}}

\def \P {\mathbb{P}}

\newcommand{\OOO}{{\mathscr{O}}}

\newcommand{\Aut}{\operatorname{Aut}}
\newcommand{\GL}{\operatorname{GL}}

\newcommand{\Bir}{\operatorname{Bir}}

\newcommand{\xref}[1]{{\rm \ref{#1}}}

\newcommand{\Pic}{\operatorname{Pic}}

\newcommand{\red}{\operatorname{red}}

\newcommand{\Gal}{\operatorname{Gal}}

\newcommand{\Alb}{\operatorname{Alb}}

\def \ge {\geqslant}
\def \le {\leqslant}

\newtheorem*{claim*}{Claim}

\title{Finite groups of birational selfmaps of threefolds}

\author{Yuri Prokhorov}
\author{Constantin Shramov}

\thanks{This work was supported by the Russian Science Foundation under grant 14-50-00005.}

\address{
Steklov Mathematical Institute of Russian Academy of Sciences, 8 Gubkina st., Moscow, Russia, 119991
}

\email{prokhoro@mi.ras.ru, costya.shramov@gmail.com}

\begin{document}

\begin{abstract}
We classify threefolds with non-Jordan birational automorphism groups.
\end{abstract}

\maketitle
\section{Introduction}

Finite subgroups of birational automorphism groups are a classical object of
study. I.\,Dolgachev and V.\,Iskovskikh \cite{Dolgachev-Iskovskikh}
managed to classify all finite subgroups of the birational automorphism group
of a plane over an algebraically closed field of characteristic zero.
However, for an arbitrary variety $X$ it would be naive to expect
any kind of classification for finite subgroups of its birational automorphism
group~$\Bir(X)$. Thus it is reasonable to find out some general properties that hold for all
such subgroups.

As a starting point one can look at common properties of
finite subgroups of linear algebraic groups
over fields of characteristic zero.
The following result is due to
H.\,Minkowski (see e.\,g.~\cite[Theorem~5]{Serre2007}
and~\cite[\S4.3]{Serre2007})
and C.\,Jordan (see~\cite[Theorem~36.13]{Curtis-Reiner-1962}).

\begin{theorem}
\label{theorem:Jordan}
If $\Bbbk$ is a number field, then there is a constant
$B=B(n, \Bbbk)$ such that for any finite subgroup
$G\subset\GL_n(\Bbbk)$ one has $|G|\le B$.
If $\Bbbk$ is an arbitrary field of characteristic zero, then
there is a constant~\mbox{$J=J(n)$} such that for any
finite subgroup~\mbox{$G\subset\GL_n(\Bbbk)$}
there exists
a normal abelian subgroup~\mbox{$A\subset G$} of index at most $J$.
\end{theorem}

This leads to the following definition

\begin{definition}[{see e.g. \cite[Definition~1]{Popov-Diff}}]
\label{definition:Jordan}
We say that a group $\Gamma$
\emph{has bounded finite subgroups}
if there exists a constant $B=B(\Gamma)$ such that
for any finite subgroup
$G\subset\Gamma$ one has $|G|\leqslant B$.
A group~$\Gamma$ is called \emph{Jordan}
(alternatively, we say
that~$\Gamma$ \emph{has Jordan property})
if there is a constant~$J$ such that
for any finite subgroup $G\subset\Gamma$ there exists
a normal abelian subgroup $A\subset G$ of index at most~$J$.
\end{definition}

It was noticed by J.-P.\,Serre that Jordan property sometimes holds
for groups of birational automorphisms.

\begin{theorem}[{\cite[Theorem~5.3]{Serre2009}, \cite[Th\'eor\`eme~3.1]{Serre-2008-2009}}]
\label{theorem:Serre}
The group $\Bir(\P^2)$ over an arbitrary field of characteristic zero is Jordan.
\end{theorem}

It appeared that one can generalize Theorem~\ref{theorem:Serre} to higher dimensions.

\begin{theorem}[{see \cite[Theorem~1.8]{Prokhorov-Shramov-2013} and~\cite[Theorem~1.1]{Birkar}}]
\label{theorem:Pr-Shr}
Let $X$ be a variety over a field of characteristic zero.
Then the following assertions hold.
\begin{itemize}
\item[(i)] If the irregularity of $X$ equals zero, then the group
$\Bir(X)$ is Jordan. In particular, this holds if $X$ is rationally
connected.

\item[(ii)] If $X$ is not uniruled, then the group
$\Bir(X)$ is Jordan.

\item[(iii)] If the irregularity of $X$ equals zero and $X$ is not uniruled,
then the group $\Bir(X)$ has bounded finite subgroups.
\end{itemize}
\end{theorem}

Finite subgroups of birational automorphism groups of projective spaces
and some other varieties
were intensively studied from the point of view of their boundedness,
see e.g.~\cite{Prokhorov-Shramov-JCr3}, \cite{Prokhorov-Shramov-p-groups}, \cite{Yasinsky2016a},
and references therein.
However, Theorem~\ref{theorem:Pr-Shr} tells us nothing about
varieties that have a structure of a conic bundle over,
say, an abelian variety. To treat this case T.\,Bandman and Yu.\,Zarhin
proved the following result.

\begin{theorem}[{\cite[Corollary 4.11]{BandmanZarhin2015a}}]
\label{theorem:Zarhin-conic}
Let $\KK$ be a field of characteristic zero
containing all roots of~$1$. Let $C$ be a conic over $\KK$.
Assume that $C$ is not $\KK$-rational, i.e.
that~\mbox{$C(\KK)=\varnothing$}. Then any finite subgroup of~\mbox{$\Aut(C)$}
has order at most~$4$.
\end{theorem}

The main purpose of this note
is to prove a result
which is a two-dimensional counterpart of Theorem~\ref{theorem:Zarhin-conic}.

\begin{theorem}\label{theorem:rational-surface-vs-BFS}
Let $\KK$ be a field of characteristic zero
containing all roots of~$1$, and $S$ be a geometrically rational surface over $\KK$.
Assume that $S$ is not $\KK$-rational but has a smooth $\KK$-point.
Then the group $\Bir(S)$ has bounded finite subgroups.
\end{theorem}

The assumption that $S$ has a $\KK$-point is crucial for Theorem~\ref{theorem:rational-surface-vs-BFS}.
Indeed, if $\KK$ is a field of characteristic zero
containing all roots of~$1$, and $C$ is a conic over $\KK$ without $\KK$-points,
then the surface $S=\P^1\times C$ has no $\KK$-points
(and thus is not $\KK$-rational), but the group~\mbox{$\Aut(S)$} contains arbitrarily large
finite cyclic subgroups.

It is known (see~\cite{Zarhin10}, or Theorem~\ref{theorem:Zarhin-AxP1} below) that there are surfaces
whose birational automorphism groups are not Jordan; they are birational to products $E\times\P^1$,
where $E$ is an elliptic curve.
The following result of V.\,Popov classifies surfaces with non Jordan birational automorphism
groups.

\begin{theorem}[{\cite[Theorem~2.32]{Popov2011}}]
\label{theorem:Popov}
Let $S$ be a surface over an algebraically closed 
field of characteristic zero.
Then the group 
$\Bir(S)$ of birational automorphisms of~$S$
is not Jordan if and only if~$S$ is 
birational to~\mbox{$E\times\P^1$}, where $E$ is an
elliptic curve. 
\end{theorem}

Applying Theorems~\ref{theorem:Zarhin-conic}
and~\ref{theorem:rational-surface-vs-BFS}, we can immediately obtain a
partial generalization of Theorem~\ref{theorem:Popov} to dimension~$3$, proving that
a threefold with a non Jordan group of birational automorphisms must be birational
to a product of~$\P^1$ and some surface. However, some
additional information about automorphism groups of non uniruled
surfaces allows us to give a complete
classification of treefolds with non Jordan birational automorphism groups.
Namely, we prove the following.

\begin{theorem}\label{theorem:dim-3}
Let $X$ be a threefold over an algebraically
closed field of characteristic zero. 
Then the group $\Bir(X)$ is not Jordan
if and only if~$X$ is birational either to $E\times\P^2$,
where $E$ is an elliptic curve, or to $S\times\P^1$, where $S$
is one of the following:
\begin{itemize}
\item
an abelian surface;

\item
a bielliptic surface;

\item
a surface of Kodaira dimension $1$ such that the Jacobian fibration
of the pluricanonical fibration $\phi\colon S\to B$ is locally trivial
\textup(in Zariski topology\textup).
\end{itemize}
\end{theorem}

\begin{remark}
Let $S$ be a smooth minimal surface of Kodaira
dimension $1$, and $\phi\colon S\to B$
be its pluricanonical fibration.
Suppose that the corresponding Jacobian fibration is locally trivial. 
If $\phi$ has a section, then it is locally trivial itself.
In particular, this implies that $g(B)\ge 2$. However, if $\phi$
does not have a section, then the genus of $B$ may be arbitrary, see
Example~\ref{example:arbitrary-base} below.
\end{remark}

The plan of the paper is as follows.
In \S\ref{section:rational-surfaces} we prove Theorem~\ref{theorem:rational-surface-vs-BFS}.
In \S\ref{section:non-rational-surfaces}
we collect some auxiliary facts about automorphism groups
of minimal non uniruled surfaces. Finally, in~\S\ref{section:threefolds} we prove
Theorem~\ref{theorem:dim-3}.

\smallskip
We are grateful to S.\,Gorchinskiy, S.\,Rybakov, K.\,Oguiso,
and L.\,Taelman for useful discussions.

\smallskip
Until the end of the paper all varieties are assumed to be projective
and defined over an algebraically closed
field $\Bbbk$ of characteristic zero if the converse is not stated explicitly.

\section{Geometrically rational surfaces}
\label{section:rational-surfaces}

In this section we
prove Theorem~\ref{theorem:rational-surface-vs-BFS}. Until the end of the section we
always assume that $\KK$ is a field of characteristic zero
that contains all roots of~$1$.

Recall that a Fano--Mori model $S/B$ of a surface~$\bar{S}$ is a smooth surface
$S$ birational to~$\bar{S}$ endowed with a morphism
$\phi\colon S\to B$ with connected fibers, where $B$ is either a curve
or a point, and the relative Picard rank~\mbox{$\mathrm{rk\,Pic}(S/B)$} equals~$1$.

The following results were proved in a series of works \cite{Manin-Cubic-forms-e-I}, \cite{Iskovskih1967e},
\cite{Iskovskikh-1970}, \cite{Iskovskikh-deg4-1973}, \cite{Iskovskikh-1979s-e};
see \cite{Iskovskikh-Factorization-1996e} for the modern approach.

\begin{theorem}\label{rationality-criterion}
Let $S/B$ be a Fano--Mori model of a geometrically rational surface over a field $\KK$.
Assume that $S(\KK)\neq \varnothing$.
Then $S$ is $\KK$-rational if and only if $K_S^2\ge 5$.
In particular, if $S/B$ is an arbitrary (not necessarily relatively
minimal) geometrically rational conic bundle or del Pezzo surface, and
$K_S^2\ge 5$, then $S$ is $\KK$-rational.
\end{theorem}

\begin{theorem}[{\cite[Theorem~1.6(iii)]{Iskovskikh-Factorization-1996e}}]
\label{k-conic-bundles}
Let $S/B$ and $S'/B'$ be Fano--Mori models of a geometrically rational surface over a field $\KK$
and let $\chi\colon S\dashrightarrow S'$ be a $\KK$-birational map.
Assume that one has~\mbox{$K_S^2\le 0$},
both $B$ and $B'$ are one-dimensional.
Then~\mbox{$K_S^2=K_{S'}^2$}.
\end{theorem}

\begin{lemma}
\label{lemma-automorphisms-del-Pezzo-surfaces}
Let $S$ be a del Pezzo surface over an arbitrary field $\Bbbk$.
If $K_S^2\le 5$, then
$$
|\Aut (S)|\le 696\, 729\, 600.
$$
\end{lemma}
\begin{proof}
We may assume that $\Bbbk$ is algebraically closed.
Since $K_S^2\le 5$, the action of
$\Aut(S)$ on $\Pic(S)$ is faithful and so
the order of the group $\Aut(S)$ is bounded by the order of the Weyl
group $W(E_8)$, see \cite[Corollary 8.2.40]{Dolgachev-ClassicalAlgGeom}, \cite[Theorem 4.5]{Manin-1967}.
\end{proof}

Now we prove Theorem~\xref{theorem:rational-surface-vs-BFS}.

\begin{proof}[Proof of Theorem~\xref{theorem:rational-surface-vs-BFS}] 
Let $G\subset \Bir(S)$ be a finite group.
We may assume that $S$ is smooth. 
Replace $S$ with its $G$-minimal model \cite{Iskovskikh-1979s-e}.
Then $S$ is either a del Pezzo surface
(not necessarily minimal if we discard the action of $G$),
or a conic bundle (again
not necessarily relatively minimal if we discard the action of $G$). 
Since~$S$ is not $\KK$-rational and~\mbox{$S(\KK)\neq \varnothing$}, we have~\mbox{$K_{S}^2\le 4$}
by Theorem~\ref{rationality-criterion}.
If~$S$ is a del Pezzo surface, then by Lemma~\ref{lemma-automorphisms-del-Pezzo-surfaces}
the order of $G$ is bounded by an absolute constant.
From now on we assume that $S$ has a conic bundle
structure~\mbox{$f\colon S\to B$}.

Denote by $\bar \KK$ the algebraic closure of $\KK$, and 
for any object $\square$ defined over $\KK$ let
$$
\bar \square=\square\otimes \bar \KK
$$
be the corresponding extension of scalars.
Let $\Delta\subset B$ be the discriminant locus of the conic bundle
$f\colon S\to B$.
Let $\bar F_1,\dots,\bar F_n$ be the fibers over~$\bar{\Delta}$.
Every fiber $\bar F_i$ has the form~\mbox{$\bar F_i^{(1)}+\bar F_i^{(2)}$},
where
$\bar F_i^{(1)}$ and $\bar F_i^{(2)}$ are $(-1)$-curves meeting transversally at one point.
Up to permutation we may assume that $\bar F_1,\dots,\bar F_m$ are
fibers whose irreducible components are $\Gal(\bar \KK/\KK)$-conjugate, and
$\bar F_{m+1},\dots,\bar F_n$ are ones
whose irreducible components are not conjugate.
Thus we have a decomposition of $\Delta$ into a disjoint union of
two $G$-invariant subsets $\Delta'$ and $\Delta''$, where
$\Delta'= f(\bar F_1+\ldots+\bar F_m)$ and
$\Delta''= f(\bar F_{m+1}+\ldots+\bar F_n)$.
Now run the Minimal Model Program (without group action) on $S$ over $B$:
\[
\xymatrix@R=7pt{
S\ar[dr]_{f}\ar[rr] && S'\ar[dl]^{f'}
\\
&B&
}
\]
This means that we contract components of the fibers $\bar F_{m+1},\dots,\bar F_n$ (one in each fiber).
We end up with a conic bundle $f'\colon S'\to B$ whose discriminant locus coincides with $\Delta'$.
(Note that the action of $G$ on $S'$ is not regular any more!)
Since~$S'$ is not $\KK$-rational and~\mbox{$S'(\KK)\neq \varnothing$}, we have
$K_{S'}^2\le 4$ by Theorem~\ref{rationality-criterion}.
Hence, one has
\[
m=|\bar \Delta'|=8-K_{S'}^2\ge 4.
\]

The group $G$ fits into the following exact sequence
\begin{equation*}
1 \longrightarrow G_F \longrightarrow G \longrightarrow G_B \longrightarrow 1,
\end{equation*}
where $G_B$ acts faithfully on the base $B$,
and $G_F$ acts faithfully on the generic fiber of~$f$.
Since~$S$ is not $\KK$-rational,
the generic fiber~$S_\eta$ of~$f$ is a conic over the field $\KK(B)$
that has no $\KK(B)$-points. Hence by Theorem~\ref{theorem:Zarhin-conic} the order of
$G_F$ is at most $4$.
On the other hand, the group $G_B$ preserves the set $\Delta'\subset B$
that consists of $m\ge 4$ points. Therefore its order is bounded
by
$$
m != (8-K_{S'}^2)!
$$
Hence
\begin{equation}\label{eq:8-factorial}
|G|\le 4(8-K_{S'}^2)!
\end{equation}
If $K_{S'}^2\ge 1$, then~\eqref{eq:8-factorial} bounds the order of $G$
by a constant that does not depend on $S'$. On the other hand,
if $K_{S'}^2\le 0$, then by Theorem~\ref{k-conic-bundles} the number~$K_{S'}^2$, and thus also the number~$m$, is a birational invariant of~$S$, i.e. it does not depend on the
choice of the birational model~$S'$. Therefore, the order of~$G$ is again bounded
by~\eqref{eq:8-factorial}.
\end{proof}

\section{Non-rational surfaces}
\label{section:non-rational-surfaces}

In this section we collect some results about automorphism groups
of non uniruled surfaces. We believe that they are well known to experts,
but in some cases we failed to find appropriate references, and thus
included the proofs for the reader's convenience.

\begin{lemma}[{cf. \cite[Exercise~IX.7(1)]{Beauville-book}}]
\label{lemma:no-rational-curves}
Let $S$ be a smooth 
minimal surface, and~\mbox{$\phi\colon S\to \P^1$} be an elliptic fibration.
Suppose that there exists a
fiber $F$ of $\phi$ such that~$F_{\red}$ is not a smooth elliptic curve.
Then the irregularity of $S$ is zero.
\end{lemma}
\begin{proof}
Every irreducible component of $F$ is a rational curve, see e.g.~\mbox{\cite[\S\,V.7]{BPVdV}}.
Hence $F$ is contracted by the Albanese morphism
$\alpha\colon S\to\Alb(S)$. Therefore, all other fibers of $\phi$
are contracted
by $\alpha$ as well, which means that $\alpha$ factors through $\phi$.
Thus the image $\alpha(S)$
is dominated by $\Alb(\P^1)$, which is a point.
\end{proof}

The following result is a version of Mordell--Weil theorem
over function fields, known as Lang--N\'eron theorem; see e.g.
\cite[Theorem~7.1]{Conrad}, and also~\cite[\S2]{Conrad}.

\begin{theorem}\label{theorem:MW}
Let $\phi\colon S\to B$ be an elliptic fibration over a curve
with a section,
and let~$\mathcal{E}$ be the fiber of $\phi$ over the general schematic
point of $B$.
Suppose that~$\phi$ is not locally trivial.
Then the group of $\Bbbk(B)$-points of~$\mathcal{E}$ is finitely
generated, and in particular the torsion subgroup of the group of points of~$\mathcal{E}$ is finite.
\end{theorem}

We will say that a group \emph{has unbounded finite subgroups}
if it fails to have bounded finite subgroups.

\begin{lemma}\label{lemma:Kodaira-dimension-1}
Let $S$ be a smooth minimal surface of Kodaira dimension $1$,
and let~\mbox{$\phi\colon S\to B$} be its pluricanonical fibration.
Then the group $\Aut(S)$ has unbounded finite subgroups
if and only if
the Jacobian fibration $\phi_J$ of~$\phi$ is locally trivial.
\end{lemma}
\begin{proof}
The morphism $\phi$ is $\Aut(S)$-equivariant.
Consider the exact sequence
$$
1\longrightarrow \Aut(S)_{\phi}\longrightarrow\Aut(S)\longrightarrow\Gamma\longrightarrow 1,
$$
where $\Gamma$ is a subgroup of $\Aut(B)$,
and the action of $\Aut(S)_{\phi}$ is fiberwise with respect
to~$\phi$.

The group $\Aut(S)_{\phi}$ is isomorphic to a subgroup of the group $\Aut(S_\eta)$,
where $S_\eta$ is the fiber of $\phi$ over the general schematic
point of $B$. Moreover, since the surface $S$ is minimal,
its birational automorphisms are biregular, and thus
$\Aut(S)_{\phi}$ is actually isomorphic to~\mbox{$\Aut(S_\eta)$}.
Since~$S_\eta$ is a smooth curve of genus $1$ over the field
$\Bbbk(B)$, we conclude that $\Aut(S_\eta)$ has a subgroup $\Delta$
of finite index isomorphic to the group of $\Bbbk(B)$-points
of the Jacobian of~$S_\eta$. If $\phi_J$ is locally trivial, then 
$\Delta$ is obviously an infinite group.
If $\phi_J$ is not locally trivial, then $\Delta$ has bounded finite subgroups
by Theorem~\ref{theorem:MW}.

Therefore, to prove the lemma it is enough to check that
$\Gamma$ always has bounded finite subgroups.
In particular, this holds if $g(B)\ge 2$,
since the group $\Aut(B)$ is finite in this case. Thus we will assume that~\mbox{$g(B)\le 1$}.

Suppose that $g(B)=1$. If $\phi$ has at least one degenerate fiber, then
the group~$\Gamma$ is finite. Thus we may assume that
$\phi$ has no degenerate fibers.
Then the Jacobian fibration $\phi_J$ of $\phi$ has no degenerate fibers as well.
Since $\phi_J$ has a section, it gives a well defined morphism $\nu$ from $B$ to the (coarse) 
moduli space $\mathcal{M}_1$ of elliptic curves. The image of $\nu$ is a single point, 
because~$\mathcal{M}_1$ is affine. Thus all fibers of $\phi_J$ are isomorphic.
This means that all fibers
of $\phi$ are isomorphic as well. Hence the surface $S$ is either
abelian, or bielliptic, see~\cite[\S\,V.5B]{BPVdV}. Both of the latter have Kodaira
dimension~$0$, which gives a contradiction.

Therefore, we see that $g(B)=0$.
Suppose that $\phi$ has a fiber $F$ such that $F_{\red}$ is not a smooth elliptic curve.
Then the irregularity of $S$ equals zero by Lemma~\ref{lemma:no-rational-curves}.
Since~$S$ is not uniruled, Theorem~\ref{theorem:Pr-Shr}(iii)
implies that the group $\Aut(S)$ has bounded finite subgroups.

Therefore, we may assume that all (set-theoretic) fibers of $\phi$ are smooth elliptic curves;
in particular, this applies to set-theoretic fibers $F_{\mathrm{red}}$, where $F$ is a multiple fiber.
We may assume that $\Bbbk=\CC$.
Then the topological Euler characteristic~\mbox{$\chi_{\mathrm{top}}(S)$} equals~$0$.
By the Noether formula one has
\[
\chi(\OOO_S)=\frac{1}{12}\left(K_S^2+\chi_{\mathrm{top}}(S)\right)=0.
\]
By the canonical bundle formula (see e.g. \cite[Theorem~V.12.1]{BPVdV})
we have
\[
K_S\sim\phi^*\left(K_B+L+ \sum (1-1/m_i) P_i\right),
\]
where $P_i$ are images of all multiple fibers of $\phi$, the fiber $\phi^{-1}(P_i)$ is a multiple fiber of multiplicity $m_i$, and $L$ is a divisor
of degree $\chi(\OOO_S)=0$. Since $S$ has Kodaira dimension~$1$,
we see that
\begin{equation*}
\deg \left(K_B+L+ \sum (1-1/m_i) P_i\right)> 0.
\end{equation*}
This implies that $\sum (1-1/m_i)\ge 2$.
Hence $\phi$ has at least three multiple fibers.
This means that the group $\Gamma$ is finite.
\end{proof}

\begin{remark}\label{remark:bielliptic}
Let $E$ be an elliptic curve, and $B$ be an arbitrary curve.
Let a finite group~$G$ act on $E$ by translations, and also act faithfully on $B$.
Consider a surface~\mbox{$S=(E\times B)/G$}. There is an elliptic fibration
$\phi\colon S\to B'$, where $B'=B/G$. Since every translation given
by a point of $E$ commutes with $G$, we see that the group of points of $E$
acts on $S$ so that the fibration $\phi$ is equivariant with respect to this action.
In particular, Theorem~\ref{theorem:MW} implies that the Jacobian fibration of $\phi$
is locally trivial. Another way to see this is to note that
the action of $G$ on the Jacobian fibration of the projection $E\times B\to B$
is via its action on~$B$.
\end{remark}

\begin{lemma}\label{lemma:non-uniruled-surface}
Let $S$ be a non uniruled surface.
Then the group $\Bir(S)$ has unbounded finite subgroups 
if and only if
$S$ is birational either to an abelian surface, or to a bielliptic surface,
or to a surface of Kodaira dimension $1$ such that the Jacobian fibration
of the pluricanonical fibration $\phi\colon S\to B$ is locally trivial.
\end{lemma}
\begin{proof}
By Theorem \ref{theorem:Pr-Shr}(iii) we may assume
that the irregularity of $S$ is positive.
Replacing~$S$ by a minimal model (of its resolution of
singularities),
we may assume that $S$ is either
an abelian surface, or a bielliptic surface, or a surface of Kodaira dimension~$1$,
or a surface of general type; also, we have~\mbox{$\Bir(S)=\Aut(S)$}.
The automorphism group of an abelian surface 
obviously has unbounded finite subgroups.
The same holds for a bielliptic surface by Remark~\ref{remark:bielliptic}.
If $S$ has Kodaira dimension $1$, then the assertion follows from
Lemma~\ref{lemma:Kodaira-dimension-1}.
Finally, if $S$ is a surface of general type, then $\Aut(S)$ is finite.
\end{proof}

The only source of varieties with non Jordan birational automorphism
groups that we are aware of is the following construction of Yu.\,Zarhin.

\begin{theorem}[\cite{Zarhin10}]
\label{theorem:Zarhin-AxP1}
Let $A$ be a (positive dimensional) abelian variety
over an arbitrary field $\KK$ of characteristic zero,
and $A'$ be a torsor over $A$.
Put~\mbox{$X=A'\times\P^1$}.
Suppose that all torsion points of $A_{\bar{\KK}}$
are defined over~$\KK$ (in particular, this implies that $\KK$ contains all 
roots of~$1$). Then the group $\Bir(X)$ is not Jordan.
\end{theorem}
\begin{proof}[Sketch of the proof]
Let $L$ be an ample line bundle on $A'$, and
$Y$ be its total space. Then $Y$ is birational
to $X$. Choose a positive integer $n$, and consider
the group
$$
G_n\cong(\ZZ/n\ZZ)^{2\dim A}
$$
of $n$-torsion points of~$A$.
The group $G_n$ acts on $A'$ by translations, and also acts on~\mbox{$\Pic(A')$}.
Replacing $L$ by its power if necessary, we may assume that
$L$ is $G_n$-invariant.
The group~$G_n$ has an extension
$$
1\longrightarrow \ZZ/n\ZZ\longrightarrow \tilde{G}_n\longrightarrow G_n\longrightarrow 1
$$
acting on $Y$,
and thus acting by birational automorphisms of $X$.
Moreover, the group~$\tilde{G}_n$ does not contain abelian subgroups
of index less than~$n$, see~\cite[\S3]{Zarhin10}. Going through this construction for arbitrarily large~$n$,
one concludes that the group $\Bir(X)$ is not Jordan.
We refer the reader to \cite{Zarhin10} for details.
\end{proof}

Using Theorem~\ref{theorem:Zarhin-AxP1}, we obtain
the following result.

\begin{corollary}\label{corollary:not-Jordan}
Let $X$ be a variety birational to a product
$S\times \P^1$, where $S$ is either a bielliptic surface,
or a surface of Kodaira dimension $1$ such that the Jacobian fibration
of the pluricanonical fibration $\phi\colon S\to B$ is locally trivial.
Then the group $\Bir(X)$ is not Jordan.
\end{corollary}
\begin{proof}
In both cases
there is an elliptic fibration $\phi\colon S\to B$
for some curve $B$; moreover, the corresponding Jacobian fibration $\phi_J$ is
locally trivial, cf. Remark~\ref{remark:bielliptic}.
Let $\mathcal{E}$ be the fiber of $\phi_J$ over the general schematic point of $B$.
Then $\mathcal{E}$ is an elliptic curve over the field $\KK=\Bbbk(B)$,
and all torsion points of $\mathcal{E}_{\bar{\KK}}$
are defined over~$\KK$.  
Furthermore, there is a fibration $\hat{\phi}\colon X\to B$
whose fiber $X_\eta$ over the general schematic point
of $B$ is isomorphic to a product $\mathcal{E}'\times\P^1$, where $\mathcal{E}'$
is a curve of genus~$1$ that is a torsor over $\mathcal{E}$;
in other words,~$\mathcal{E}$ is the Jacobian of $\mathcal{E}'$.
By Theorem~\ref{theorem:Zarhin-AxP1} the group $\Bir(X_\eta)$
is not Jordan. Since $\Bir(X_\eta)$ is a subgroup
of $\Bir(X)$, we conclude that the group $\Bir(X)$ is also not Jordan.
\end{proof}

An example of a surface of Kodaira dimension $1$ with the properties
required in Corollary~\ref{corollary:not-Jordan} is a product
$B\times E$, where $B$ is a curve of genus at least $2$ and $E$
is an elliptic curve. The following example shows that there are
much more surfaces of this kind.

\begin{example}\label{example:arbitrary-base}
Let $B'$ be a smooth curve of genus $g(B')\ge 2$ with an automorphism $\theta$ of order
$n$. Denote by $\bar{G}$ the subgroup in $\Aut(B)$ generated by $\theta$, and put
$B=B'/\bar{G}$.
Let $E$ be an elliptic curve, and $e$ be its point of order $n$.
Let the generator of a cyclic group $G\cong\ZZ/n\ZZ$ act on
$B'\times E$ as
$$
(x,y)\longmapsto (\theta(x),y+e).
$$
Put $S=(B'\times E)/G$. Then there is an elliptic fibration
$\phi\colon S\to B$ whose general fiber is isomorphic to $E$.
Comparing the canonical bundle formula
for $\phi$ with the Hurwitz formula for the finite cover $B'\to B$,
we see that the canonical class $K_S$ is a pull-back of some
$\QQ$-divisor of positive degree on $B$. In particular, $K_S$ is nef,
$S$ is a minimal surface of Kodaira dimension $1$, and $\phi$ is
its pluricanonical fibration. If the action of the group~$\bar{G}$
on $B'$ is not free, then $\phi$ has multiple fibers, and in particular
$\phi$ is not locally trivial. On the other hand, the Jacobian
fibration of $\phi$ is locally trivial by Remark~\ref{remark:bielliptic}.
Note that one can
arrange such situation for an arbitrary curve $B$, including rational
and elliptic curves. For instance, one can produce $B'$ as a double cover
of $B$ with sufficiently many branch points, and choose $\theta$ to be the Galois
involution of this double cover.
\end{example}

\section{Threefolds}
\label{section:threefolds}

In this section we prove Theorem~\ref{theorem:dim-3}.

\begin{definition}[{\cite[Definition~2.5]{Prokhorov-Shramov-2013},
\cite[Definition~1.1]{BandmanZarhin2015a}}]
\label{definition:bounded-rank}
We say that a group~$\Gamma$
\emph{has finite subgroups of bounded rank}
if there exists a constant $R=R(\Gamma)$ such that
each finite abelian subgroup
$A\subset\Gamma$ is generated by at most $R$ elements.
\end{definition}

\begin{lemma}
\label{lemma:group-theory}
Let
$$
1\longrightarrow\Gamma'\longrightarrow\Gamma\longrightarrow\Gamma''
$$
be an exact sequence of groups. Then the following assertions
hold.
\begin{itemize}
\item[(i)] If $\Gamma'$ is Jordan, and $\Gamma''$ has bounded finite
subgroups, then $\Gamma$ is Jordan.

\item[(ii)] If $\Gamma'$ has bounded finite subgroups,
and $\Gamma''$ is Jordan and has
finite subgroups of bounded
rank, then $\Gamma$ is Jordan.
\end{itemize}
\end{lemma}
\begin{proof}
For assertion~(i) see~\cite[Lemma~2.3]{Prokhorov-Shramov-2013}.
For assertion~(ii) see~\cite[Lemma~2.8]{Prokhorov-Shramov-2013}.
\end{proof}

\begin{proposition}
\label{proposition:non-uniruled}
Let $X$ be a non uniruled variety. Then $\Bir(X)$ has
finite subgroups of bounded rank.
\end{proposition}
\begin{proof}
See the proof of~\cite[Corollary~3.8]{BandmanZarhin2015a},
or~\cite[Remark~6.9]{Prokhorov-Shramov-2013}.
\end{proof}

Recall that to any variety $X$
one can associate the \emph{maximal rationally connected fibration}
\[
\phi_{\mathrm{RC}}\colon X\dasharrow X_{\mathrm{nu}},
\]
which is a canonically defined rational map with
rationally connected fibers and non-uniruled base $X_{\mathrm{nu}}$
(see~\cite[\S\,IV.5]{Kollar-1996-RC},
\cite[Corollary~1.4]{Graber-Harris-Starr-2003}).
The maximal rationally connected fibration is equivariant
with respect to the group~$\Bir(X)$.

\begin{theorem}[{\cite[Theorem~1.5]{BandmanZarhin2015a}}]
\label{theorem:Zarhin-rel-dim-1}
Let $X$ be a variety, and $\phi\colon X\dasharrow Y$
be the maximal rationally connected fibration.
Suppose that $\dim Y=\dim X-1$. Then $\Bir(X)$
is Jordan unless $X$ is birational to $Y\times\P^1$.
\end{theorem}

\begin{corollary}\label{corollary:dim-3-rel-dim-1}
Let $X$ be a threefold, and $\phi\colon X\dasharrow Y$
be the maximal rationally connected fibration.
Suppose that $\dim Y=2$. Then $\Bir(X)$
is not Jordan if and only if $X$ is birational to $Y'\times\P^1$,
where $Y'$ is either an abelian surface, or
a bielliptic surface, or
a surface of Kodaira dimension $1$ such that the Jacobian fibration
of the pluricanonical fibration $\phi\colon S\to B$ is locally trivial.
\end{corollary}
\begin{proof}
By Theorem~\ref{theorem:Zarhin-rel-dim-1} we may assume
that $X$ is birational to $Y\times\P^1$.
Since $\phi$ is equivariant with respect to the group
$\Bir(X)$, we have an exact sequence
$$
1\longrightarrow \Bir(X)_{\phi}\longrightarrow\Bir(X)\longrightarrow \Bir(Y),
$$
where the action of $\Bir(X)_{\phi}$ is fiberwise with respect to $\phi$.
The group $\Bir(X)_{\phi}$ is isomorphic to $\Aut(\P^1_{\Bbbk(Y)})$,
and thus it is Jordan by Theorem~\ref{theorem:Pr-Shr}(i).

By construction the surface
$Y$ is not uniruled. We know from Lemma~\ref{lemma:non-uniruled-surface}
that~$\Bir(Y)$ has bounded finite
subgroups unless $Y$ is birational either to an abelian surface, or to a bielliptic surface,
or to a surface of Kodaira dimension $1$ such that the Jacobian fibration
of the pluricanonical fibration $\phi\colon S\to B$ is locally trivial.
If $Y$ is birational to none of the latter surfaces, then
the group $\Bir(X)$ is Jordan by Lemma~\ref{lemma:group-theory}(i).
If on the contrary~$Y$ is of one of these three types, then
the group~$\Bir(X)$ is not Jordan by Theorem~\ref{theorem:Zarhin-AxP1} and Corollary~\ref{corollary:not-Jordan}.
\end{proof}

\begin{lemma}
\label{lemma:rel-dim-2}
Let $X$ be a variety, and $\phi\colon X\dasharrow Y$
be the maximal rationally connected fibration.
Suppose that $\dim Y=\dim X-2$, and that $\phi$ has a rational section. Then $\Bir(X)$
is Jordan unless $X$ is birational to $Y\times\P^2$, and $\Bir(Y)$ has unbounded finite subgroups.
\end{lemma}
\begin{proof}
Let $S$ be the fiber of $\phi$ over the general schematic point of $Y$.
Then $S$ is a geometrically rational surface defined over the field
$\KK=\Bbbk(Y)$, and $S$ has a smooth $\KK$-point by assumption.
Since $\phi$ is equivariant with respect to $\Bir(X)$, we have an exact sequence
$$
1\longrightarrow \Bir(X)_{\phi}\longrightarrow \Bir(X)\longrightarrow\Bir(Y),
$$
where the action of $\Bir(X)_{\phi}$ is fiberwise with respect to $\phi$.

Suppose that $X$ is not birational to $Y\times\P^2$. This means that
$S$ is not rational over~$\KK$.
The group $\Bir(X)_{\phi}$ is isomorphic to the group
$\Bir(S)$, and thus has bounded finite subgroups by Theorem~\ref{theorem:rational-surface-vs-BFS}.
On the other hand, the variety $Y$ is not uniruled. Thus the group $\Bir(Y)$ is Jordan by
Theorem~\ref{theorem:Pr-Shr}(ii) and has finite subgroups
of bounded rank by Proposition~\ref{proposition:non-uniruled}. Hence
the group $\Bir(X)$ is Jordan by Lemma~\ref{lemma:group-theory}(ii).

Therefore, we see that $X$ is birational to $Y\times\P^2$.
The group $\Bir(X)_{\phi}$ is isomorphic to $\Bir(\P^2_{\Bbbk(Y)})$,
and thus it is Jordan by Theorem~\ref{theorem:Pr-Shr}(i).
This means that if $\Bir(Y)$ has bounded finite subgroups, then
the group $\Bir(X)$ is Jordan by Lemma~\ref{lemma:group-theory}(i).
\end{proof}

\begin{corollary}\label{corollary:dim-3-rel-dim-2}
Let $X$ be a threefold, and $\phi\colon X\dasharrow Y$
be the maximal rationally connected fibration.
Suppose that $\dim Y=1$. Then $\Bir(X)$
is not Jordan if and only if $X$ is birational to $Y'\times\P^2$,
where $Y'$ is an elliptic curve.
\end{corollary}
\begin{proof}
The map $\phi$ has a rational section by~\cite{Graber-Harris-Starr-2003}.
Thus by Lemma~\ref{lemma:rel-dim-2} we may assume
that~$X$ is birational to $Y\times\P^2$, and the group~$\Bir(Y)$ is infinite.
Since $Y$ is a non-rational curve, the assertion immediately follows from
Lemma~\ref{lemma:rel-dim-2} and Theorem~\ref{theorem:Zarhin-AxP1}.
\end{proof}

Now we are ready to prove Theorem~\ref{theorem:dim-3}.

\begin{proof}[Proof of Theorem~\textup{\ref{theorem:dim-3}}]
Let $\phi\colon X\dasharrow Y$ be the maximal rationally connected fibration.
If~\mbox{$\dim Y=0$}, then $X$
is rationally connected, so that $\Bir(X)$ is Jordan by
Theorem~\ref{theorem:Pr-Shr}(i).
If $\dim Y=3$, then $X$
is not uniruled, so that $\Bir(X)$ is Jordan by
Theorem~\ref{theorem:Pr-Shr}(ii).
If~\mbox{$\dim Y=2$}, then the assertion follows from
Corollary~\ref{corollary:dim-3-rel-dim-1}.
Finally, if $\dim Y=1$,
then the assertion follows from Corollary~\ref{corollary:dim-3-rel-dim-2}.
\end{proof}

We conclude the paper with the following question.

\begin{question}\label{question:rational}
Let $X$ be a rationally connected threefold over an algebraically closed
field of characteristic zero. Suppose that $X$ is not rational. Is it true that
the group $\Bir(X)$ has bounded finite subgroups?
\end{question}

At the moment we do not have any reasonable expectation
about the answer to Question~\ref{question:rational}. If the answer appears to be positive, proving this
may require some delicate work with automorphism groups of Fano varieties, including
singular ones (cf.~\cite{Prokhorov-GFano-2}, \cite{Prokhorov-planes}, \cite{PrzyjalkowskiShramov2016}).
Also, at the moment we do not know the answer to Question~\ref{question:rational}
in the case of a smooth cubic threefold, and believe that working out this example may
be very instructive.

\def\cprime{$'$}



\begin{thebibliography}{BPVdV84}

\bibitem[Bea78]{Beauville-book}
Arnaud Beauville.
\newblock {\em Surfaces alg\'ebriques complexes}.
\newblock Soci\'et\'e Math\'ematique de France, Paris, 1978.
\newblock Avec une sommaire en anglais, Ast{\'e}risque, No. 54.

\bibitem[Bir16]{Birkar}
Caucher Birkar.
\newblock Singularities of linear systems and boundedness of {F}ano varieties.
\newblock {\em ArXiv e-print}, 1609.05543, 2016.

\bibitem[BPVdV84]{BPVdV}
W.~Barth, C.~Peters, and A.~Van~de Ven.
\newblock {\em Compact complex surfaces}, volume~4 of {\em Ergebnisse der
 Mathematik und ihrer Grenzgebiete (3) [Results in Mathematics and Related
 Areas (3)]}.
\newblock Springer-Verlag, Berlin, 1984.


\bibitem[BZ17]{BandmanZarhin2015a}
Tatiana Bandman and Yuri~G. Zarhin.
\newblock Jordan groups, conic bundles and abelian varieties.
\newblock {\em Algebraic Geometry}, 4(2):229--246, 2017.

\bibitem[Con06]{Conrad}
Brian Conrad.
\newblock Chow's {$K/k$}-image and {$K/k$}-trace, and the {L}ang-{N\'e}ron
 theorem.
\newblock {\em Enseign. Math. (2)}, 52(1-2):37--108, 2006.

\bibitem[CR62]{Curtis-Reiner-1962}
Charles~W. Curtis and Irving Reiner.
\newblock {\em Representation theory of finite groups and associative
 algebras}.
\newblock Pure and Applied Mathematics, Vol. XI. Interscience Publishers, a
 division of John Wiley \& Sons, New York-London, 1962.

\bibitem[DI09]{Dolgachev-Iskovskikh}
Igor~V. Dolgachev and Vasily~A. Iskovskikh.
\newblock Finite subgroups of the plane {C}remona group.
\newblock In {\em Algebra, arithmetic, and geometry: in honor of {Y}u. {I}.
 {M}anin. {V}ol. {I}}, volume 269 of {\em Progr. Math.}, pages 443--548.
 Birkh\"auser Boston Inc., Boston, MA, 2009.

\bibitem[Dol12]{Dolgachev-ClassicalAlgGeom}
Igor~V. Dolgachev.
\newblock {\em Classical algebraic geometry}.
\newblock Cambridge University Press, Cambridge, 2012.

\bibitem[GHS03]{Graber-Harris-Starr-2003}
Tom Graber, Joe Harris, and Jason Starr.
\newblock Families of rationally connected varieties.
\newblock {\em J. Am. Math. Soc.}, 16(1):57--67, 2003.

\bibitem[Isk67]{Iskovskih1967e}
V.~A. Iskovskikh.
\newblock Rational surfaces with a pencil of rational curves.
\newblock {\em Mat. USSR Sbornik}, 3(4):563--587, 1967.

\bibitem[Isk70]{Iskovskikh-1970}
V.~A. Iskovskikh.
\newblock Rational surfaces with a pencil of rational curves and with positive
 square of the canonical class.
\newblock {\em Mat. Sb.}, 83(1):90--119, 1970.

\bibitem[Isk73]{Iskovskikh-deg4-1973}
V.A. Iskovskikh.
\newblock Birational properties of a surface of degree {$4$} in {$\mathbb
 P_k^4$.}
\newblock {\em Math. USSR, Sb.}, 17:30--36, 1973.

\bibitem[Isk80]{Iskovskikh-1979s-e}
V.~A. Iskovskikh.
\newblock Minimal models of rational surfaces over arbitrary fields.
\newblock {\em Math. USSR-Izv.}, 14(1):17--39, 1980.

\bibitem[Isk96]{Iskovskikh-Factorization-1996e}
V.~A. Iskovskikh.
\newblock Factorization of birational mappings of rational surfaces from the
 point of view of {M}ori theory.
\newblock {\em Russian Math. Surveys}, 51(4):585--652, 1996.

\bibitem[Kol96]{Kollar-1996-RC}
J{\'a}nos Koll{\'a}r.
\newblock {\em Rational curves on algebraic varieties}, volume~32 of {\em
 Ergebnisse der Mathematik und ihrer Grenzgebiete. 3. Folge. A Series of
 Modern Surveys in Mathematics [Results in Mathematics and Related Areas. 3rd
 Series. A Series of Modern Surveys in Mathematics]}.
\newblock Springer-Verlag, Berlin, 1996.

\bibitem[Man67]{Manin-1967}
Yu.~I. Manin.
\newblock Rational surfaces over perfect fields. {II}.
\newblock {\em Mat. Sb. (N.S.)}, 72 (114):161--192, 1967.

\bibitem[Man74]{Manin-Cubic-forms-e-I}
Yu.~I. Manin.
\newblock {\em Cubic forms: algebra, geometry, arithmetic}.
\newblock North-Holland Publishing Co., Amsterdam, 1974.
\newblock Translated from the Russian by M. Hazewinkel, North-Holland
 Mathematical Library, Vol. 4.

\bibitem[Pop11]{Popov2011}
Vladimir~L. Popov.
\newblock On the {M}akar-{L}imanov, {D}erksen invariants, and finite
 automorphism groups of algebraic varieties.
\newblock In {\em Peter Russell's Festschrift, Proceedings of the conference on
 Affine Algebraic Geometry held in Professor Russell's honour, 1--5 June 2009,
 McGill Univ., Montreal.}, volume~54 of {\em Centre de Recherches
 Math\'ematiques CRM Proc. and Lect. Notes}, pages 289--311, 2011.

 \bibitem[Pop16]{Popov-Diff}
Vladimir~L. Popov.
\newblock Finite subgroups of diffeomorphism groups.
\newblock {\em Proc. Steklov Inst. Math.}, 289(1):221--226, 2016.

\bibitem[Pro13]{Prokhorov-GFano-2}
Yuri Prokhorov.
\newblock G-{F}ano threefolds, {II}.
\newblock {\em Adv. Geom.}, 13(3):419--434, 2013.

\bibitem[Pro15]{Prokhorov-planes}
Yu. Prokhorov.
\newblock On {$G$}-{F}ano threefolds.
\newblock {\em Izv. Ross. Akad. Nauk Ser. Mat.}, 79(4):159--174, 2015.

\bibitem[PS14]{Prokhorov-Shramov-2013}
Yu. Prokhorov and C.~Shramov.
\newblock Jordan property for groups of birational selfmaps.
\newblock {\em Compositio Math.}, 150(12):2054--2072, 2014.

\bibitem[PS16a]{Prokhorov-Shramov-JCr3}
Yu. Prokhorov and C.~Shramov.
\newblock Jordan constant for {C}remona group of rank 3.
\newblock {\em ArXiv e-print}, 1608.00709, 2016, to appear in Moscow Math. J.

\bibitem[PS16b]{Prokhorov-Shramov-p-groups}
Yu. Prokhorov and C.~Shramov.
\newblock p-subgroups in the space {C}remona group.
\newblock {\em ArXiv e-print}, 1610.02990, 2016.

\bibitem[PS16c]{PrzyjalkowskiShramov2016}
Victor Przyjalkowski and Constantin Shramov.
\newblock Double quadrics with large automorphism groups.
\newblock {\em Proc. Steklov Inst.}, 294, 2016.

\bibitem[Ser07]{Serre2007}
Jean-Pierre Serre.
\newblock Bounds for the orders of the finite subgroups of {$G(k)$}.
\newblock In {\em Group representation theory}, pages 405--450. EPFL Press,
 Lausanne, 2007.

\bibitem[Ser09]{Serre2009}
Jean-Pierre Serre.
\newblock A {M}inkowski-style bound for the orders of the finite subgroups of
 the {C}remona group of rank 2 over an arbitrary field.
\newblock {\em Moscow Math. J.}, 9(1):193--208, 2009.

\bibitem[{Ser}10]{Serre-2008-2009}
Jean-Pierre {Serre}.
\newblock Le groupe de {C}remona et ses sous-groupes finis.
\newblock In {\em {S\'eminaire Bourbaki. Volume 2008/2009. Expos\'es
 997--1011}}, pages 75--100, ex. Paris: Soci\'et\'e Math\'ematique de France
 (SMF), 2010.

\bibitem[Yas16]{Yasinsky2016a}
Egor Yasinsky.
\newblock The {J}ordan constant for {C}remona group of rank $2$.
\newblock {\em ArXiv e-print}, 1610.09654, 2016, to appear in Bull. Korean Math. Soc.

\bibitem[Zar14]{Zarhin10}
Yuri~G. Zarhin.
\newblock Theta groups and products of abelian and rational varieties.
\newblock {\em Proc. Edinburgh Math. Soc.}, 57(1):299--304, 2014.

\end{thebibliography}
\end{document}